\documentclass{amsart}
\numberwithin{equation}{section}
\author{Dennis Presotto \and Michel Van den Bergh}
\title[Noncommutative versions of classical birational transformations.]{Noncommutative versions of some classical birational transformations.}
\usepackage{verbatim}
\usepackage[all]{xy}
\usepackage{amscd}
\usepackage{amssymb}
\usepackage{mathabx}
\usepackage{tikz}

\usepackage{mathtools}

\usetikzlibrary{matrix,arrows,decorations.pathmorphing}

\let\cal\mathcal
\def\Ascr{{\cal A}}

\def\Dscr{{\cal D}}
\def\Escr{{\cal E}}
\def\Fscr{{\cal F}}

\def\Iscr{{\cal I}}
\def\Jscr{{\cal J}}

\def\Lscr{{\cal L}}
\def\Mscr{{\cal M}}
\def\Nscr{{\cal N}}
\def\Oscr{{\cal O}}

\let\blb\mathbb

\def \PP{{\blb P}}

\def \ZZ{{\blb Z}}

\def \NN{{\blb N}}

\def\id{\text{id}}
\def\Id{\operatorname{id}}

\def\Bimod{\operatorname{Bimod}}

\def\Mod{\operatorname{Mod}}

\def\Gr{\operatorname{Gr}}

\def\QGr{\operatorname{QGr}}

\def\Qch{\operatorname{Qch}}

\def\Ext{\operatorname {Ext}}
\def\Hom{\operatorname {Hom}}

\def\im{\operatorname {im}}

\def\coker{\operatorname {coker}}
\def\ker{\operatorname {ker}}

\def\Tor{\operatorname {Tor}}

\def\id{{\operatorname {id}}}

\def\Pic{\operatorname {Pic}}

\def\r{\rightarrow}

\DeclareMathOperator{\Proj}{Proj}

\DeclareMathOperator{\Alg}{Alg}

\DeclareMathOperator{\Tors}{Tors}

\newtheorem{lemma}{Lemma}[section]

\newtheorem{theorem}[lemma]{Theorem}

\theoremstyle{definition}

\newtheorem{definition}[lemma]{Definition}

{

}

\theoremstyle{remark}

\newtheorem{remark}[lemma]{Remark}

\newdimen\uboxsep \uboxsep=1ex
\def\uboxn#1{\vtop to 0pt{\hrule height 0pt depth 0pt\vskip\uboxsep
\hbox to 0pt{\hss #1\hss}\vss}}

\def\uboxs#1{\vbox to 0pt{\vss\hbox to 0pt{\hss #1\hss}
\vskip\uboxsep\hrule height 0pt depth 0pt}}

\newcommand\blfootnote[1]{%
  \begingroup
  \renewcommand\thefootnote{}\footnote{#1}%
  \addtocounter{footnote}{-1}%
  \endgroup
}

\def\BIMOD{\operatorname{BIMOD}}

\def\Frac{\operatorname{Frac}}

\begin{document}
\maketitle
\blfootnote{The first author is an aspirant of the FWO, the second author is a senior researcher of the FWO.}
\begin{abstract}
In this paper we generalize some classical birational 
transformations to the non-commutative case. In particular
we show that 3-dimensional quadratic Sklyanin algebras
(non-commutative projective planes) and 3-dimensional cubic Sklyanin
algebras (non-commutative quadrics) have the same function
field. In the same vein we construct an analogue
of the Cremona transform for non-commutative projective planes. 
\end{abstract}

\tableofcontents

\section{Introduction}
Below $k$ is an algebraically closed field. 
Artin-Schelter regular algebras were introduced in
\cite{artinschelter} and subsequently classified in dimension three \cite{artinschelter, ATV1, stephenson}.
Throughout we will only consider three-dimensional AS-regular algebras generated in degree one.
For such algebras $A$
there are two possibilities:
\begin{enumerate}
\item $A$ is generated by three elements satisfying three quadratic
  relations (the ``quadratic case'').  In this case $A$ has Hilbert
  series $1/(1-t)^3$, i.e.\ the same Hilbert series as a polynomial
  ring in three variables. 
\item $A$ is generated by two elements satisfying two cubic relations
(the ``cubic case''). In this case $A$ has Hilbert series $1/(1-t)^2(1-t^2)$.
\end{enumerate}
For use below we define $(r,s)$ to be respectively the number of generators of $A$ and the degrees of the 
relations. Thus $(r,s)=(3,2)$ or $(2,3)$ depending on whether $A$ is quadratic 
or cubic. 

If $B=k+B_1+B_2+\dots$ is an $\NN$-graded ring satisfying suitable conditions 
then we can associate a non-commutative scheme $\Proj B$ to it whose 
category of quasi-coherent sheaves is defined to be $\QGr(B)\overset{\text{def}}{=} \Gr(B)/\Tors(B)$
where $\Gr(B)$ is the category of right $B$-modules   and $\Tors(B)$
is the category graded right $B$-modules that have locally right bounded grading \cite{artinzhang}.
When $A$ is a quadratic three-dimensional AS-regular algebra then $\Proj A$ may be thought off
as a non-commutative plane. Similarly if $A$ is cubic then $\Proj A$ may be viewed as a non-commutative quadric. The rationale
for this is explained in \cite{VdB38}. 

The
classification of three-dimensional AS-regular algebras $A$
is in terms of suitable geometric data $(Y,\sigma,\Lscr)$ where $Y$ is
a $k$-scheme, $\sigma$ is automorphism of $Y$ and $\Lscr$ is
a line bundle on $Y$.

%
More precisely:
in the quadratic case $Y$ is either $\PP^2$ (the ``linear case'') 
or $Y$ is embedded as a
divisor of degree 3 in $\PP^2$  (the ``elliptic case'')
 and $\Lscr$ is the restriction of $\Oscr_{\PP^2}(1)$. In the cubic case
$Y$ is either $\PP^1 \times \PP^1$ (the ``linear case'')
or $Y$ is embedded as a
divisor of bidegree (2,2) in $\PP^1 \times \PP^1$ (the ``elliptic case'')
and $\Lscr$ is the restriction of $\Oscr_{\PP^1\times \PP^1}(1,0)$.
The geometric data must also satisfy an additional numerical condition which we will not discuss
here. 

Starting from the geometric data $(Y,\sigma,\Lscr)$ we construct a so-called ``twisted homogeneous
coordinate ring'' $B=B(Y,\sigma,\Lscr)$.
 It is an $\NN$-graded ring such that 
\begin{equation}
\label{twisted}
B_n = \Gamma(Y,\Lscr \otimes
\Lscr^{\sigma} \otimes \ldots \otimes \Lscr^{\sigma^{n-1}})
\end{equation}
with product $a\cdot b=a\otimes b^{\sigma^n}$ for $|a|=n$. The corresponding
AS-regular algebra $A=A(Y,\sigma,\Lscr)$ is obtained from $B$ by dropping all relations
in degree $>s$. By virtue of the construction there is a graded surjective $k$-algebra
homomorphism $A\r B$ and this is an isomorphism in the linear case
and it has a kernel generated by a normal element $g$ in degree $s+1$ in the elliptic case.

According to \cite{AV} there is an equivalence of categories $\QGr(B) \cong \Qch(Y)$. In our current language this can be written
as 
\[
\Proj B\cong Y
\]
So  the non-commutative scheme $X=\Proj A$ contains the commutative scheme $Y$ 
(via the surjection $A\r B$). In the linear case $X=Y$, and in the quadratic case
$Y$ is a so-called ``divisor'' in $X$ \cite[Section 3.6]{VdB19}.

If $Y$ is a smooth elliptic curve, $\sigma$ is a translation such that $\sigma^{s+1}\neq \id$ 
and $\Lscr$ is a line bundle of degree $r$
then we
call the corresponding AS-regular algebra a \emph{Sklyanin algebra}. In that case
the normal element $g$ is actually central.
Since
any two line bundles of the same degree on a smooth elliptic curve are related
by a translation, which necessarily commutates with $\sigma$, it is easy to see 
that the resulting Sklyanin algebra depends up to isomorphism only on $(E,\sigma)$. So we sometimes
drop $\Lscr$ from the notation.
Furthermore $\Proj A$ does not change if we compose $\sigma$ with a translation
by a point of order $s+1$ (See for example \cite[\S 8]{ATV2}). In other words $\Proj A$ depends only on $\sigma^{s+1}$.

A three-dimensional AS-regular algebra $A$ is a noetherian domain and in particular 
it has a graded field of fractions $\Frac(A)$ in which we invert all non-zero
homogeneous elements of $A$. The part of degree zero $\Frac_0(A)$ of $\Frac(A)$ 
will be called the \emph{function field of $\Proj A$}.

In this note we prove the following result announced in \cite{VdBSt}. A similar result by Rogalski-Sierra-Stafford was announced in \cite{Sierratalk}.
\begin{theorem} \label{mainth1} If $A$, $A'$ are a cubic and a quadratic
  Sklyanin algebra respectively with
  geometric data $(Y,\sigma)$ and $(Y,\psi)$ such that $\sigma^3=\psi^4$. Then $\Proj A$
  and $\Proj A'$ have the same function field.
\end{theorem}

The proof of this result is geometric. In the commutative case the passage
from $\PP^1\times \PP^1$ to $\PP^2$ goes by blowing up a point $p$ and then contracting 
the strict transforms of the two rulings through this point.  
One may short circuit this construction by considering a suitable
linear system on $\PP^1\times \PP^1$ with base point in $p$.
It is this construction that we generalize first. To do this we have to step outside the category of graded algebras
and  work in the slightly larger category of $\ZZ$-algebras (additive categories whose objects are indexed by $\ZZ$, see \S\ref{secremas} below).

 So what we will actually do is
the following: let $A$ be a cubic Sklyanin algebra and let $A^{(2)}$ be its
$2$-Veronese with the corresponding $\ZZ$-algebra being denoted  by $\check{A}^{(2)}$. Associated to a point $p\in Y$ we will construct a sub-$\ZZ$-algebra $D$ of $\check{A}^{(2)}$ 
which is 3-dimensional quadratic Artin-Schelter $\ZZ$-algebra in the sense of \cite{VdB38}. Again invoking \cite{VdB38} this $\ZZ$-algebra must correspond
to a 3-dimensional quadratic Artin-Schelter \emph{graded algebra} $A'$.
It will turn out that the geometric data of $A$ and $A'$ are related as in
 Theorem \ref{mainth1}. 
Note that  the use of $\ZZ$-algebras is essential here as there is no direct embedding $A' \hookrightarrow A^{(2)}$ 
of graded rings. 

Another classical birational transformation is the so-called ``Cremona transform''. 
It is obtained by blowing up the tree vertices of a triangle and then contracting
the sides. In this note we will also show that the Cremona transform has a non-commutative
version  and that it is yields an automorphism of the function field of a three-dimensional
quadratic Sklyanin algebra. The properties of this automorphism will be discussed elsewhere.


In \S\ref{secblowup} we explain how in the non-commutative case the approach via linear systems
is related to
the  blowup construction introduced 
in \cite{VdB19}. 

\begin{remark}
A more ring-theoretic approach to blowups of noncommutative surfaces was taken by Rogalski-Sierra-Stafford in \cite{Noncblowup}. They also used this technique in their companion paper \cite{Classifyingorders} to classify certain orders in a generic 3-dimensional Sklyanin algebra.
\end{remark}

\medskip

\begin{remark} Cubic 3-dimensional Artin-Schelter regular algebras are
  a special case of the non-commutative quadrics introduced in
  \cite{VdB38}. Theorem \ref{mainth1} generalizes to such quadrics but
  the proof becomes slightly more technical. For this reason we have chosen to
  write down the proof of Theorem \ref{mainth1} separately.
\end{remark}

\section{Reminder on AS-regular $\ZZ$-algebras}
\label{secremas}
For background material on $\ZZ$-algebras see \cite{Sierra} and also  sections 3 and 4 of \cite{VdB38}. Recall
that a ($k$-)$\ZZ$-algebra is defined as a $k$-algebra $A$ (without
unit) with a decomposition $\displaystyle A = \bigoplus_{(m,n)
  \in \ZZ^2} A_{m,n}$ such that the multiplication satisfies
$A_{m,n}A_{n,j} \subset A_{m,j}$ and $A_{m,n}A_{i,j}=0$ if $n \neq
i$. Moreover we require the existence of local units $e_n\in A_{nn}$ satisfying $e_m x = x = x
e_n$ whenever $x \in A_{m,n}$. The category of $\ZZ$-algebras is
denoted by Alg($\ZZ$). Every graded $k$-algebra $A$ gives rise to a
$\ZZ$-algebra $\check{A}$ via $\check{A}_{m,n} = A_{n-m}$. 
Most graded notions have a natural $\ZZ$-algebra counterpart. 
For example we say that $A \in $ Alg($\ZZ)$ is positively graded if
$A_{m,n}=0$ whenever $m > n$. 
A $\ZZ$-algebra over $k$ is said to be connected, if
it is positively graded, each $A_{m,n}$ is finite dimensional over $k$
and $A_{m,m} \cong k$ for all $m$.

If $A \in \Alg(\ZZ$) then we say $M$ is a graded right-$A$-module if
it is a module in the usual sense together with a decomposition $M =
\oplus_n M_n$ satisfying $M_m A_{m,n} \subset M_n$ and $M_m A_{i,n} =
0$ if $i \neq m$. We denote the category of graded $A$-modules by
Gr($A$). (Obviously $\Gr(A)=\Gr(\check{A})$ if $A$ is a graded ring.)
If $A$ is a connected $\ZZ$-algebra over $k$ we denote the graded
$A$-modules $P_{n,A} = e_n A$ and $S_{n,A} \cong k$ is the unique
simple quotient of $P_{n,A}$. 

\begin{definition}
A $\ZZ$-algebra $A$ over $k$ is said to be AS-regular if the following conditions are satified:
\begin{enumerate}
\item $A$ is connected
\item $\dim_k(A_{m,n})$ is bounded by a polynomial in $n-m$
\item The projective dimension of $S_{n,A}$ is finite and bounded by a number independent of $n$
\item $\displaystyle \forall n \in \NN: \sum_{i,j} \dim_k \left( \Ext_{\Gr(A)}^i(S_{j,A},P_{n,A}) \right) =1$
\end{enumerate}
\end{definition}
It is immediate that if a graded algebra $A$ is AS-regular in the sense of \cite{ATV1}, then $\check{A}$ is AS-regular in the above sense.

$\ZZ$-algebra analogues of three dimensional quadratic and cubic Artin-Schelter regular algebras were classified in \cite{VdB38} (following \cite{Bondal} in the quadratic case). 
We will describe the quadratic case as this is the only case we will need.
In this case the
  classification is in terms of triples $(Y,\Lscr_0,\Lscr_1)$ where
  $Y$ is either a (possibly singular, non-reduced) curve of arithmetic genus 1 (the ``elliptic
  case'') or $Y=\PP^2$ (the ``linear case'') and $\Lscr_0,\Lscr_1$ are line
  bundles of degree 3 on $Y$ such that $\Lscr_0\not\cong \Lscr_1$ in the elliptic 
case and $\Lscr_0=\Lscr_1=\Oscr_{\PP^2}(1)$ in the linear case. The triple must 
satisfy some other technical conditions which  are however vacuous in the case that
$Y$ is a smooth elliptic curve. 

To construct a $\ZZ$-algebra from this data we first introduce the ``elliptic helix'' $(\Lscr_i)_{i\in \ZZ}$
associated to $(\Lscr_0,\Lscr_1)$. This is a collection of line bundles satisfying
\[
\Lscr_i\otimes_{\Oscr_Y} \Lscr_{i+1}^{-2}\otimes_{\Oscr_Y} \Lscr_{i+2}=\Oscr_Y
\]
We put $V_i=H^0(Y,\Lscr_i)$ and 
\[
R_i=\ker(H^0(Y,\Lscr_i)\otimes H^0(Y,\Lscr_{i+1})\r H^0(Y,\Lscr_i\otimes_{\Oscr_Y} \Lscr_{i+1}))
\]
By definition the quadratic AS-regular $\ZZ$-algebra $A=A(Y,\Lscr_0,\Lscr_1)$ associated to $(Y,\Lscr_0,\Lscr_1)$ is generated by $V_i(=A_{i,i+1})$ subject to the relations $R_i\subset V_i\otimes V_{i+1}$.
The ``Hilbert function'' of $A$ is 
\begin{equation}
\dim A_{m,m+a}=
\begin{cases}
\frac{(a+1)(a+2)}{2}&\text{if $a\ge 0$}\\
0&\text{if $a<0$}
\end{cases}
\end{equation}

Using the line bundles $(\Lscr_i)_i$ be may define a $\ZZ$-algebra analogue  $B=B(Y,(\Lscr_i)_i)$ of a
twisted homogeneous coordinate ring (see the introduction) where
\[
B_{m,n}=\Gamma(Y,\Lscr_m \otimes \ldots \otimes \Lscr_{n-1} )
\]
If $A$ is the 3-dimensional AS-regular $\ZZ$-algebra constructed above
then there is a surjective map
\[
\phi:A\r B
\]
where $A$ is obtained from $B$ by dropping all relations in degree $(m,n)$ for $n\ge m+s+1$.

If $A$ is 3-dimensional quadratic AS-regular algebra with geometric
data $(Y,\sigma,\Lscr)$ then the elliptic helix corresponding to
$\check{A}$ is $(\Lscr^{\sigma^i})_i$. This follows immediately from
the construction of $A$ from $(Y,\sigma,\Lscr)$ as given in
\cite{ATV1} (see the introduction for an outline).
\section{Non-commutative geometry}
It will be convenient to use the formalism of non-commutative geometry
used in \cite{VdB19} which we summarize here. For more details we
refer to loc.\ cit.. See also \cite{Smith}. We will change the terminology and notations slightly to be
more compatible with current conventions.  

For us a non-commutative
scheme will be a Grothendieck category (i.e.\ an abelian category with
a generator and exact filtered colimits). To emphasize that we think
of non-commutative schemes as geometric objects, we denote them by
roman capitals $X$, $Y$, \dots. When we refer to the category represented by a non-commutative scheme $X$ then we write $\Qch(X)$. 

A morphism $\alpha:X\r Y$ between non-commutative schemes will be a right exact
functor $\alpha^\ast:\Qch(Y)\r \Qch(X)$ possessing a right adjoint
(denoted by $\alpha_\ast$). In this way the non-commutative schemes form a
category (more accurately: a two-category).

In this paper we often view commutative schemes as non-commutative schemes. 
More precisely
if $X$ is a commutative scheme, then 
$\Qch(X)$ will be the category of
quasi-coherent sheaves on $X$. It is proved in \cite{EnochEstrada} that
this is a Grothendieck category. Furthermore $X$ can be recovered from $\Qch(X)$
\cite{Brandenburg,Gabriel,rosenberg1}.

If $X$ is a non-commutative scheme then we think of objects in $\Qch(X)$ as sheaves of right
modules on $X$. To define the analogue of a sheaf of algebras on $X$ however we
need a category of bimodules on $X$ (see \cite{Translation} for the case where $X$ is commutative).
The most obvious way to proceed is to define the category $\Bimod(X-Y)$ of $X-Y$-bimodules as the right exact functors
$\Qch(X)\r \Qch(Y)$  commuting with direct limits. The action of a bimodule $\Nscr$ on an object $\Mscr\in \Qch(X)$ is
written as $\Mscr\otimes_X \Nscr$.

If we define the ``tensor product'' of bimodules as composition then we can define
algebra objects on $X$ as algebra objects in the category of $X-X$-bimodules  and in this we may extend
much of the ordinary commutative formalism. For example the identity functor $\Qch(X)\r \Qch(X)$ 
is a natural analogue of the structure sheaf, and as such it will be denoted by $o_X$. If $\Ascr$ is an algebra object on
$X$ then it is routine to
define an abelian category $\Mod(\Ascr)$ of right-$\Ascr$-modules. We have $\Mod(o_X)=\Qch(X)$. Unraveling
all the definitions it turns out that $-\otimes_X-$ (the ``tensor product'' (composition) in the 
monoidal category $\Bimod(X-X)$) and $-\otimes_{o_X}-$ (the tensor product over the
algebra $o_X$) have the same meaning. We will use both notations, depending on the context.

Unfortunately $\Bimod(X-Y)$ appears not to be an abelian category and this represents a  technical inconvenience
which is solved in \cite{VdB19} by embedding $\Bimod(X-Y)$ into a larger
category $\BIMOD(X-Y)$ consisting of ``weak bimodules''. The category $\BIMOD(X-Y)$ 
is opposite to the category of left exact functors $\Qch(Y)\r\Qch(X)$. Since
left exact functors are determined by their values on injectives, they
trivially form an abelian category. The category $\Bimod(X-Y)$ 
is the full category of $\BIMOD(X-Y)$ consisting of functors having
a left adjoint. Or equivalently: functors commuting with direct products.  

This being said, these technical complication will be invisible in this paper as all bimodules we encounter
will be in $\Bimod(X-Y)$.  

If $A$ be a graded algebra then the associated non-commutative scheme $X = \Proj A$
is defined by $\Qch(X)=\QGr(A)=\Gr(A)/\Tors(A)$, as discussed above. Note that
$\Proj A$ is only reasonably behaved when $A$ satisfies suitable homological conditions.
See \cite{artinzhang,Polishchuk}.
 We denote the
quotient functor $\Gr(A) \rightarrow \QGr(A)$ by $\pi$. The object $\pi A$ is
denoted by $\Oscr_X$. The ``shift by $n$'' functor $\Qch(X)$ is written as $\Mscr\mapsto \Mscr(n)$
and the corresponding bimodule is written as $o_X(n)$. In particular $o_X=o_X(0)$ and $\Oscr_X(n) = \Oscr \otimes_{o_X}o_X(n) = \pi(A(n))$.

\section{Construction of the subalgebra $D$ of $\check{A}^{(2)}$}
\label{secconstruction}
We devote the rest of the paper to the proof of Theorem \ref{mainth1} as well as the  construction
of the non-commutative Cremona transform. The treatment of both constructions will be almost entirely parallel. 
So let $A$ be a $3$-dimensional Sklyanin algebra, which may be either quadratic or cubic, and put $X=\Proj A$. 

As explained in the introduction 
(see also \cite{ATV1}) $A$ corresponds to a triple $(Y,\sigma,\Lscr)$, where $Y$ is smooth elliptic curve, $\sigma$ is a translation
and $\Lscr$ is a line bundle of degree $r$ on $Y$. The relation is given by the fact 
there is a regular central element $g \in A_{s+1}$  such
that $A/(g) =
B(Y,\sigma,\Lscr)$ where $B=B(Y,\sigma,\Lscr)$ is a so-called ``twisted homogeneous coordinate ring'' (see \eqref{twisted}). 

Using the resulting equivalence of categories (see the introduction and \cite{AV})
\[
\Proj B\cong Y
\]
we will write $o_Y(n)\in \Bimod(Y-Y)$ for the shift by $n$-functor on $\Proj B$. Then we
have
\[
o_Y(1)=\sigma_\ast(-\otimes_{\Oscr_Y} \Lscr)
\]
(the tensor product takes place in the category of sheaves of $Y$-modules).

The inclusion functor $\Qch(Y)\subset \Qch(X)$ (i.e.\ the functor dual to the graded
algebra morphism $A\r B$) has a left adjoint which
we denote by $-\otimes_{o_X} o_Y$ (on the level of graded modules
it corresponds to tensoring by $A/gA$).  Note that
in this way $o_Y$ is viewed as a $X-Y$-bimodule.

Below we will routinely regard a sheaf of $\Oscr_Y$-modules $\Nscr$ as an object in $\Bimod(Y-Y)$
by identifying it with the functor $-\otimes_{\Oscr_Y}\Nscr$. It is easy to see that the
resulting functor
\begin{equation}
\label{exact}
\Qch(Y)\r\Bimod(Y-Y)\subset \BIMOD(Y-Y)
\end{equation}
is fully faithful and exact. 

Similarly we regard an $Y-Y$-bimodule
$\Mscr$ as an $X-X$-bimodule by defining the corresponding
functor to be
\[
\Qch(X)\xrightarrow{-\otimes_{o_X} o_Y} \Qch(Y)\xrightarrow{-\otimes_{o_Y} \Mscr} \Qch(Y)\hookrightarrow \Qch(X)
\]
In this way $o_Y$ becomes an $X-X$-bimodule and one checks that it is in fact an algebra
quotient of $o_X$. Note that $o_Y$ now denotes both an algebra on $X$ and an algebra
on $Y$ (the identity functor) but for both interpretations we have $\Mod(o_Y)\cong\Qch(Y)$.  

For use in the sequel we write 
\[
o_X(-Y)=\ker (o_X\r o_X)
\]
$o_X(-Y)$ is the ideal in $o_X$ corresponding to the graded ideal
$gA\subset A$. Note that since $g$ is central we have in fact $o_X(-Y)=o_X(-3)$.

%
If $\Mscr\in \Qch(X)$ then we define the ``global sections'' of $\Mscr$ as
\[
\Gamma(X,\Mscr)=\Hom_X(\Oscr_X,\Mscr)
\]
Similarly we define the global sections of an $X-X-$bimodule $\Nscr$ as in
\cite[Section 3.5]{VdB19}:
\[ 
\Gamma(X,\Nscr) := \Hom(\Oscr_X, \Oscr_X \otimes_{o_X} \Nscr) 
\]
Use of the functor $\Gamma(X,-)$ on bimodules requires some
care since it is apriori not left exact. However in our applications
it will be. 

Note that $\Nscr$ is  an algebra object in the category of
bimodules then $\Gamma(X,\Nscr)$ is in fact an algebra for purely formal
reasons.  The same holds true for graded algebras and $\ZZ$-algebras.

It is easy to see that $A_n$ is equal to the global sections of $o_X(n)$:
\begin{eqnarray*}
\Gamma ( X, o_X(n)) & := & \Hom_X(\Oscr_X, \Oscr_X(n)) \\
 & = & \Hom_{\QGr(A)}( \pi(A), \pi(A(n))) \\
 & = & \Hom_{\Gr(A)}( A, A(n)) \qquad \text{ \cite[Theorem 8.1(5)]{artinzhang}}\\
 & = & A_n
\end{eqnarray*}
where the third equality follows from the AS-regularity of $A$.
Thus for the $\ZZ$-algebra associated to the two-Veronese of $A$ we have:
\[ \check{A}^{(2)}_{m,n} = \Gamma ( X, o_X(2(n-m))) = \Gamma( X, o_X(-2m) \otimes_{o_X} o_X(2n)) \]


Below $(p_i)_i$ is a collection 
of 
 points  on $Y$: three distinct points in case 
$(r,s)=(3,2)$ and one point in case $(r,s)=(2,3)$. Let $d=\sum_i p_i$ be the
corresponding divisor on $Y$. As above we consider $\Oscr_d$ as a $Y-Y$-bimodule
but to avoid confusion we
write it as $o_d$.
Following our convention above we also consider $o_d$ as an $X$-bimodule.
Put
\begin{align} 
m_{d,Y}& = \ker (o_Y \rightarrow o_d) \label{mdY}\\
m_d &= \ker (o_X \rightarrow o_d) \label{md}
\end{align}
Clearly $m_{d,Y}\in \Bimod(Y-Y)$ as $m_{d,Y}$ corresponds to an ordinary ideal
sheaf in $\Oscr_Y$ (see \eqref{exact} above). The fact that $m_d\in \Bimod(X-X)$ follows by
applying \cite[Corollary\ 5.5.6]{VdB19} repeatedly for the different $p_i$.

Finally consider the following bimodules 
over $X$, respectively $Y$:
\begin{align} 
(\Dscr_Y)_{m,n}& = 
\begin{cases}
o_Y(-2m) \otimes_{o_Y} m_{\tau^{-m} d,Y} \ldots m_{\tau^{-n+1} d,Y} \otimes_{o_Y} o_Y(2n) & \text{if $n\ge m$}\\
0 & \text{if $n< m$}
\end{cases}\\
\Dscr_{m,n} &=
\begin{cases}
 o_X(-2m) \otimes_{o_X} m_{\tau^{-m} d} \ldots m_{\tau^{-n+1} d} \otimes_{o_X} o_X(2n) &\text{if $n\ge m$}\\
0&\text{if $n<m$}
\end{cases}
\label{case2}
\end{align}
where $\tau = \sigma^{s+1}$.  Here  $m_{\tau^{-k} d} \ldots m_{\tau^{-l} d}$ is the image
of
\[
m_{\tau^{-l} d}\otimes_X \cdots \otimes_X  m_{\tau^{-l} d} \longrightarrow o_X\otimes_X\cdots\otimes_X o_X=o_X
\]
A priori this image lies only in $\BIMOD(X-X)$ but with the same method as the proof of
\cite[Proposition 6.1.1]{VdB19} one verifies that it lies in fact in $\Bimod(X-X)$.

The collections of bimodules $\Dscr\overset{\text{def}}{=}\bigoplus_{m,n} \Dscr_{m,n}$,  
$\Dscr_Y\overset{\text{def}}{=}\bigoplus_{m,n}  (\Dscr_Y)_{m,n}$ represent 
$\ZZ$-algebra objects respectively in $\Bimod(X-X)$ and $\Bimod(Y-Y)$.  For example the product
\[
\Dscr_{m,n}\otimes_{o_X} \Dscr_{n,p} 
\]
is given by
\begin{multline*}
 o_X(-2m) \otimes_{o_X} m_{\tau^{-m} d} \ldots m_{\tau^{-n+1} d} \otimes_{o_X} o_X(2n)
\otimes_{o_X}  o_X(-2n) \otimes_{o_X} m_{\tau^{-n} d} \ldots m_{\tau^{-p+1} d} \otimes_{o_X} o_X(2p)
\r \\o_X(-2m) \otimes_{o_X} m_{\tau^{-m} d} \ldots m_{\tau^{-n+1} d} \otimes_{o_X} m_{\tau^{-n} d} \ldots m_{\tau^{-p+1} d} \otimes_{o_X} o_X(2p)
\r \\o_X(-2m) \otimes_{o_X} m_{\tau^{-m} d} \ldots m_{\tau^{-n+1} d}m_{\tau^{-n} d} \ldots m_{\tau^{-p+1} d} \otimes_{o_X} o_X(2p)
\end{multline*}
Denote the global sections of $\Dscr$ and $\Dscr_Y$ by $D$, $D_Y$ respectively. Thus $D$
and $D_Y$ are both $\ZZ$-algebras. 

\medskip

The inclusion $\Dscr_{m,n} \hookrightarrow o_X(2(n-m))$ gives rise to
an inclusion of $\ZZ$-algebras $D \hookrightarrow \check{A}^{(2)}$ by using 
\cite[Lemma 8.2.1]{VdB19} with $\Escr=\Oscr_X$. This is the sought sub-$\ZZ$-algebra of
 $\check{A}^{(2)}$.

%
\section{Analysis of $D_Y$}
Our aim is to show that $D$ is a quadratic AS-regular $\ZZ$-algebra.
 The first step in understanding $D$  is showing
that the quotient $\ZZ$-algebra $D_Y$ is
a $\ZZ$-analogue of a twisted homogeneous coordinate ring (see \S\ref{secremas}). We do this next.

We 
have 
 to find an elliptic helix $\{ \Lscr_i \}_i$ as defined in \S\ref{secremas} such that 
\[ 
(D_Y)_{m,n} = B(Y, \{ \Lscr_i\}_i )_{m,n} := \Gamma(Y,\Lscr_m \otimes \ldots \otimes \Lscr_{n-1} )
\]
%
The functor $- \otimes_{o_Y} m_{d,Y}$ is given by $- \otimes_{\Oscr_Y}
\Mscr_{d,Y}$ where $\Mscr_{d,Y}$ is the ideal sheaf of $d$ on $Y$ (see
\eqref{exact} above).  Moreover as we have already mentioned
$o_Y(1)=\sigma_* ( - \otimes_k \Lscr)$ (using the notations of
\cite{Translation} we could write this as: $m_{d,Y} = \
_1(\Mscr_{d,Y})_1$ and $o_Y(1) = \ _1\Lscr_\sigma$).
Using the fact that $-\otimes o_Y(2n)$ is an autoequivalence we
compute for $n\ge m$\\
\begin{equation}
\label{thc}
\begin{aligned}
  (D_Y)_{m,n} & =  \Hom \left(\Oscr_Y ,\Oscr_Y (-2m) \otimes m_{\tau^{-m}d,Y} \ldots m_{\tau^{-n+1}d,Y}  \otimes o_Y(2n) \right) \\
  & =  \Hom \left(\Oscr_Y(-2n), \Oscr_Y (-2m) \otimes m_{\tau^{-m}d,Y} \ldots m_{\tau^{-n+1}d,Y} \right) \\
  & =  \Hom \left( \left(\Lscr \otimes \ldots \otimes \Lscr^{\sigma^{2n-1}} \right)^{-1}, \left(\Lscr \otimes \ldots \otimes  \Lscr^{\sigma^{2m-1}} \right)^{-1} \otimes \Mscr_{\tau^{-m}d,Y} \ldots \Mscr_{\tau^{-n+1}d,Y} \right) \\
  & =  \Hom \left( \Oscr_Y, \Mscr_{\tau^{-m}d,Y} \ldots \Mscr_{\tau^{-n+1}d,Y} \otimes \Lscr^{\sigma^{2m}} \otimes \ldots \otimes \ \Lscr^{\sigma^{2n-1}} \right) \\
  & =  \Gamma \left( Y, \Mscr_{\tau^{-m}d,Y} \ldots \Mscr_{\tau^{-n+1}d,Y} \otimes \Lscr^{\sigma^{2m}} \otimes \ldots \otimes  \Lscr^{\sigma^{2n-1}} \right)\\
  &= \Gamma \left(Y,\Lscr_m \otimes \ldots \otimes \Lscr_{n-1} \right)
\end{aligned}
\end{equation}
with
\begin{equation}
\label{lidef}
\Lscr_i = \Mscr_{\tau^{-i}d,Y} \otimes \Lscr^{\sigma^{2i}} \otimes 
 \Lscr^{\sigma^{2i+1}} 
\end{equation}
A routine but somewhat tedious verification shows that the isomorphism
constructed in \eqref{thc} sends the product on the left 
to the obvious product on the right corresponding to the tensorproduct.



We now have to check that the $(\Lscr_i)_i$ constitute an elliptic helix as introduced in \S\ref{secremas}. Using our standing
hypothesis that $Y$ is smooth (since $A$ was assumed to be a Sklyanin algebra)
we must verify the following facts
\begin{enumerate}
\item $\deg \Lscr_i=3$
\item $\Lscr_0\not\cong\Lscr_1$.
\item $
\Lscr_i\otimes \Lscr_{i+1}^{-2}\otimes \Lscr_{i+2}\cong\Oscr_Y
$.
\end{enumerate}
We use the following lemma. 
\begin{lemma}
\label{geometricdatalemma}
If $A$ is quadratic, one has
 \[
\sigma^\ast(\Lscr_i)=\Lscr_{i+1}
\]
and if $A$ is cubic
\[
\psi^{\ast}(\Lscr_i)=\Lscr_{i+1}
\]
 where $\psi$ is an arbitrary translation satisfying $\psi^3=\sigma^4$.
\end{lemma}
\begin{proof}
We compute in the quadratic case
\begin{align*}
\sigma^\ast(\Lscr_i)\otimes \Lscr^{-1}_{i+1}&
=
 \Mscr_{\tau^{-i}d,Y}^\sigma \otimes 
 \Lscr^{\sigma^{2i+1}} \otimes \Lscr^{\sigma^{2i+2}} 
\otimes \Mscr_{\tau^{-i-1}d,Y} ^{-1}\otimes (\Lscr^{\sigma^{2i+2}})^{-1}
 \otimes ( \Lscr^{\sigma^{2i+3}} )^{-1}\\
&=
 \Mscr^\sigma_{\tau^{-i}d,Y} \otimes \Lscr^{\sigma^{2i+1}} 
\otimes (\Mscr^{\sigma^3}_{\tau^{-i}d,Y}) ^{-1} \otimes ( \Lscr^{\sigma^{2i+3}} )^{-1}
\end{align*}

  Since $\sigma$ is a
  translation there is an invertible sheaf $\Nscr$ of degree zero
  on $Y$ such that for each invertible sheaf $\Mscr$ on $Y$ we have
  the following identities in $\Pic(Y)$:
\[ 
[\sigma^* \Mscr ] = [\Mscr] + \deg(\Mscr) \cdot [\Nscr] 
\]
(This statement is true in even higher generality, see \cite[Theorem 4.2.3]{VdB38}) Thus
\[
[\sigma^\ast(\Lscr_i)\otimes \Lscr^{-1}_{i+1}]
=
\left(\deg(\Mscr_{\tau^{-i}d,Y})+\deg(\Lscr^{\sigma^{2i}})-3\deg(\Mscr_{\tau^{-i}d,Y})-3\deg(\Lscr^{\sigma^{2i}})\right)[\Nscr]
=0
\]
taking into account that in the quadratic case
\begin{align*}
\deg(\Mscr_{\tau^{-i}d,Y})&=-\deg d=-3\\
\deg(\sigma^{\ast^{2i}}\Lscr)&=3
\end{align*}

Now we consider the cubic case. It will be convenient to introduce a translation
$\sigma_3$ which is a cube root of $\sigma$
\begin{align*}
  \sigma^{4^\ast}_3(\Lscr_i)\otimes \Lscr^{-1}_{i+1}&=
  \Mscr^{\sigma_3^4}_{\tau^{-i}d,Y} \otimes \Lscr^{\sigma_3^{6i+4}}
  \otimes \Lscr^{\sigma_3^{6i+7}} \otimes
  (\Mscr_{\tau^{-i}d,Y}^{\sigma_3^{12}}) ^{-1}\otimes
  (\Lscr^{\sigma_3^{6i+6}})^{-1} \otimes  \Lscr^{\sigma_3^{6i+9}}
  )^{-1}
\end{align*}
Let $\Nscr_3$ be a line bundle of degree
zero such that for any line bundle $\Mscr$
\[
[\sigma^*_3 \Mscr ] = [\Mscr] + \deg(\Mscr) \cdot [\Nscr_3] 
\]
We obtain
\begin{multline*}
[\sigma^{4^\ast}_3(\Lscr_i)\otimes \Lscr^{-1}_{i+1}]
=\left(4\deg(\Mscr_{\tau^{-i}d,Y} )+4\deg(\Lscr^{\sigma_3^{6i}})+
7\deg(\Lscr^{\sigma_3^{6i}})-12\deg(\Mscr_{\tau^{-i}d,Y})\right.\\
\left.-6\deg(\Lscr^{\sigma_3^{6i}})-9\deg(\Lscr^{\sigma_3^{6i}}) \right)[\Nscr_3]=0
\end{multline*}
taking into account that this time
\begin{align*}
\deg(\Mscr_{\tau^{-i}d,Y})&=-\deg d=-1\\
\deg(\Lscr^{\sigma^{6i}_3})&=2 \qedhere
\end{align*}

\end{proof}
We now verify that $(\Lscr_i)_i$  is an elliptic helix. Condition (1) is immediate and condition (3) follows from
Lemma \ref{geometricdatalemma}. Assume that (2) is false in the quadratic case. Then $\sigma^\ast(\Lscr_0)=\Lscr_0$.
In other words $\sigma$ is translation by a point of order three. But this contradicts 
our assumption that $A$ is a Sklyanin algebra. Now assume that (2) is false in the cubic case. The $\psi$ is a translation
by a point of order three and from the definition of $\psi$ it follows that $\sigma$ is translation
by a point of order four, again contradicting the fact that $A$ is  Sklyanin algebra. 
\section{Showing that $D$ is AS-regular}
For use below recall some some commutation formulas. First note that since $o_Y(1)=\sigma_\ast(-\otimes_{\Oscr_Y}\Lscr)$ we have
\[
o_d\otimes_{o_Y} o_Y(1)=o_Y(1)\otimes_{o_Y} o_{\sigma d}
\]
(we may see this by applying both sides to an object in $\Qch(Y)$). Using the definitions of $m_d$, $m_{d,Y}$ (see 
\eqref{mdY}, \eqref{md}) we deduce from this
\begin{align*}
m_{d,Y}\otimes_{o_Y} o_Y(1)&=o_Y(1)\otimes_{o_Y} m_{\sigma d,Y}\\
m_d\otimes_{o_X} o_X(1)&=o_X(1)\otimes_{o_X} m_{\sigma d}
\end{align*}
Similar formulas also hold for longer products of $m$'s such as for example appear in the definition of $(\Dscr_Y)_{m,n}$
and $\Dscr_{m,n}$.

If $\Mscr$ is a bimodule then we will write 
 $(a)\Mscr$ for
  $\Oscr_X(a)\otimes_{o_X} \Mscr$. Thus the ``right structure'' of $\Mscr$ is $(0)\Mscr$.
For the sequel we need a resolution of 
$(a)\Dscr_{m,m+1}$. In the quadratic case we use the following lemma.
\begin{lemma} 
  \label{rightstructure} Let $A$ be a quadratic AS-regular algebra of
  dimension 3. Let $q_1,q_2,q_3$ be distinct non-collinear points in
  $Y$ and let $Q_1,Q_2,Q_3$ be the corresponding point
  modules.\footnote{A point module over $A$ is a graded right
    $A$-module generated in degree zero with Hilbert function
    $1,1,1,1,1,\ldots$. There is a 1-1 correspondence between points
    in $Y$ and point modules over $A$. See \cite{ATV2}.}. 
 Pick an $m$
  in $(Q_1\oplus Q_2\oplus Q_3)_0$ whose three components are non-zero
  and let $M=mA$. Then the minimal resolution of $M$ has the following
  form
\[
0\r A(-3)^{\oplus 2}\r A(-2)^{\oplus 3} \r A\r M\r 0
\]
\end{lemma}
\begin{proof}
  Let $g$ be the normalizing element of degree three in $A$ and let
  $B=A/gA$. By using the explicit category equivalence
  $\Qch(B)\cong\QGr(Y)$ \cite{AV} one easily proves that the map $B_{\ge 1}\r M_{\ge
    1}$ is surjective. Whence the corresponding map $u:A_{\ge 1}\r
  M_{\ge 1}$ is also surjective.

Look at the exact sequence
\[
0\r \ker u\r A_{\ge 1}\r
  M_{\ge 1}\r 0
\]
Tensoring this exact sequence with $k$ yields an exact sequence
\[
\Tor_1^A(M_{\ge 1},k)\r \ker u\otimes_A k \r A_{\ge 1}\otimes_A
k\xrightarrow{\bar{u}} M_{\ge 1}\otimes k \r 0
\]
Now both $A_{\ge 1}$ and $M_{\ge 1}$ are generated in degree one and
furthermore $\dim A_1=\dim M_1$. Hence it follows that $\bar u$ is an
isomorphism. Therefore $\ker u\otimes_A k$ is a quotient of
$\Tor_1^A(M_{\ge 1},k)$. From the fact that $M_{\ge 1}$ is a sum of
shifted point modules we compute that $\Tor_1^A(M_{\ge
  1},k)=k(-2)^3$. Thus $\ker u$ is a quotient of $A(-2)^3$. Now using
the fact that $M$ has no torsion and hence has projective dimension
two we may now complete the full resolution of $M$ using a Hilbert
series argument. 
\end{proof}
Note that 
\[
\Dscr_{m,m+1}=o_X(-2m) \otimes_{o_X} m_{\tau^{-m} d} \otimes_{o_X} o_X(2(m+1)) 
\]
and
thus
\[
\Oscr_{X}(a)\otimes_{o_X} \Dscr_{m,m+1}=(\Oscr_X\otimes_{o_X} m_{\sigma^{2m-a}\tau^{-m}d})(a+2)
\]
where
\[
\Oscr_X\otimes_{o_X} m_{\sigma^{2m-a}\tau^{-m}d}=\ker(\Oscr_X\r \Oscr_{\sigma^{2m-a}\tau^{-m}d})
\]
Thus $\Oscr_X\otimes_{o_X} m_{\sigma^{2m-a}\tau^{-m}d}$ is of the form $\pi (\ker (A\r M))$ 
with $M$ as in Lemma \ref{rightstructure}.
We  conclude that we have a  resolution of
$(a)\Dscr_{m,m+1}$ of the form  
\begin{equation}
\label{ref-9-9}
0\r \Oscr_X(a-1)^{\oplus 2} \r \Oscr_X(a)^{\oplus 3} \r (a)\Dscr_{m,m+1}\r 0 
\end{equation}
This resolution is actually of the form
\begin{equation}
\label{ref-10-10}
0\r \Oscr_X(a-1)^{\oplus 2} \r \Oscr_X(a)\otimes_k D_{m,m+1} \r (a)\Dscr_{m,m+1}\r 0 
\end{equation}
In the cubic  case the resolution will follow from the next lemma:
\begin{lemma}
Let $A$ be a cubic AS-regular algebra of dimension 3. Let $p$ be a point in $Y$ and let $P$ be the corresponding point module. Then there is a complex of the following form:
\begin{equation}
\label{minimal}
0\r A(-5) \xrightarrow{(\zeta,0)} A(-4)^{\oplus 2} \oplus A(-3) \r A(-2)^{\oplus 3} \r A\r P\r 0
\end{equation}
where $\zeta$ is part of the minimal resolution of $k$ 
 as given in \cite[Theorem 1.5.]{artinschelter} 
\[
0\r A(-4) \xrightarrow{\zeta} A(-3)^2\xrightarrow{\varepsilon} A(-1)^2\xrightarrow{\delta_0} A\xrightarrow{\gamma} k\r 0
\]
Moreover the complex \eqref{minimal} is exact everywhere except at $A$ 
where it has one-dimensional cohomology, concentrated in degree one. 
\end{lemma}
\begin{proof}
From \cite[Proposition 6.7.]{ATV2} we know $P$ has the following (minimal) resolution:
\[ 0 \r A(-3) \r A(-2) \oplus A(-1) \r A \r P \r 0 \]
Combining this with the minimal resolution for $k$ we get the following diagram
\begin{center}
\begin{tikzpicture}
\matrix(m)[matrix of math nodes,
row sep=3em, column sep=3em,
text height=1.5ex, text depth=0.25ex]
{ & & 0 & & & \\
  & & k(-1) & & & \\
0 & A(-3) & A(-1) \oplus A(-2) & A & P & 0 \\
& & A(-2)^{\oplus 2} \oplus A(-2) & & & \\
& & A(-4)^{\oplus 2} & & & \\ 
& & A(-5) & & & \\
& & 0 & & & \\ };
\path[->,font=\scriptsize]
(m-3-1) edge (m-3-2)
(m-3-2) edge node[above]{$\alpha$} 
(m-3-3) edge node[below]{$\exists\eta$} (m-4-3)
(m-3-3) edge node[above]{$\beta$} (m-3-4)
edge node[auto]{$\gamma$} (m-2-3)
(m-3-4) edge node[above]{$\varphi$} (m-3-5)
(m-3-5) edge (m-3-6)
(m-2-3) edge (m-1-3)
(m-4-3) edge node[auto]{$\delta_0\oplus\Id$} (m-3-3)
(m-5-3) edge node[auto]{$(\varepsilon,0)$} (m-4-3)
(m-6-3) edge node[auto]{$\zeta$} (m-5-3)
(m-7-3) edge (m-6-3);
\end{tikzpicture}
\end{center}
Put $\delta=\delta_0\oplus \id$.
The
 existence of the map $\eta$ such that $\delta\circ \eta = \alpha$ follows from the projectivity of $A(-3)$
 and the fact that $\gamma\circ\alpha$ is zero by degree reasons.
By diagram chasing one easily finds that $\ker(\beta \circ \delta) = \im(\varepsilon) \oplus \im(\eta)$ and hence we end up with the following complex:
\[ 0 \rightarrow A(-5) \xrightarrow{\zeta\oplus 0} A(-4)^{\oplus 2} \oplus
A(-3) \xrightarrow{\eta \oplus \varepsilon} A(-2)^{\oplus 3}
\xrightarrow{\beta \circ \delta} A \xrightarrow{\varphi} P \rightarrow
0 \] Using diagram chasing again one easily checks that this complex
is exact everywhere except at $A$. We then conclude with a Hilbert series argument. 
\end{proof}
%
%
In a similar way as in the quadratic case we conclude that
$(a)\Dscr_{m,m+1}$ has a resolution of the form
\begin{equation}
\label{ref-9-9-2}
0 \r \Oscr_X(a-3) \xrightarrow{(\zeta,0)} \Oscr_X(a-2)^{\oplus 2} \oplus \Oscr_X(a-1) \r \Oscr_X(a)^{\oplus 3} \r (a)\Dscr_{m,m+1} \r 0 
\end{equation}
which is actually of the form
\begin{equation}
\label{ref-10-10-2}
 0 \r \Jscr(a) \oplus \Oscr_X(a-1) \r \Oscr_X(a) \otimes_k D_{m,m+1} \r(a)\Dscr_{m,m+1} \r 0 
\end{equation} 
where
\begin{equation}
\label{jscr}
\Jscr\overset{\text{def}}{=}\coker(\Oscr_X(-3) \xrightarrow{\zeta} \Oscr_X(-2)^{\oplus 2})
\end{equation}
We will now prove some vanishing results.
An object in $\Qch(X)$ will be said to have finite length if it is a finite extension
of objects of the form $\Oscr_p$, $p\in Y$. 
Likewise an object in
$\Bimod(X-X)$ will be said to have finite length if it is a finite extension
of $o_p$ for $p\in Y$. The objects of finite length are fully understood, see \cite[Chapter 5]{VdB19}. Note that by \cite[Proposition 5.5.2]{VdB19} $o_p$
is a simple object in $\Bimod(X-X)$ so the Jordan-Holder theorem applies to finite length bimodules. 
\begin{lemma}
\label{vanishing}
A finite length object in $\Qch(X)$ has no higher cohomology.
\end{lemma}
\begin{proof} For an object of the form $o_p$ this follows from \cite[Proposition 5.1.2]{VdB19} 
with $\Fscr=\Oscr_X$. The general case follows from the long exact sequence
for $\Ext$. 
\end{proof}
\begin{lemma}
\label{ref-7.2-11}
$H^2(X,(-l)\Dscr_{m,n})=0$ for $l\le 2n-2m+s$. 
\end{lemma}
\begin{proof} We only need to consider the case $n\ge m$.
This follows from the fact that $(-l)\Dscr_{m,n}\subset \Oscr_X(2n-2m-l)$
with finite length cokernel and from the standard vanishing properties on $\QGr(A)$
(see for example \cite[Theorem 8.1]{artinzhang}).
\end{proof}
\begin{lemma}
\label{ref-7.3-12}
$H^1(X,(a)\Dscr_{m,n})=0$ for $a\ge -s+1$.
\end{lemma}
\begin{proof}
We only need to consider the case $n\ge m$.
The proof for $a \ge -1$ is similar in the cases $(r,s)=(2,3)$ and $(r,s)=(3,2)$ so we will give the proof for the first case as it is slightly longer. Afterwards we will consider the case $(r,s)=(2,3)$ and $a=-2$.

Suppose $(r,s)=(2,3)$ and $a \geq -1$.
We prove $H^1(X,(a)\Dscr_{m,n})=0$ by induction on $n-m$. As $(a)\Dscr_{m,m}=\Oscr_X(a)$ the base case follows from the standard vanishing on $X$.

For the induction step we proceed as follows:  From \cite[Theorem 5.5.10]{VdB19}
and the fact that $\Dscr_{m,n}\subset o_X(2n-2m)$ with finite length cokernel
we may deduce that the kernel of the obvious surjective map 
\[
\Dscr_{m,m+1}\otimes_X \Dscr_{m+1,n}\r \Dscr_{m,n}
\]
has finite length.
Using \cite[Lemma 8.2.1]{VdB19} we see that this remains the case if we left tensor with $\Oscr_X(a)$. Thus we obtain a short
exact sequence in $\Qch(X)$
\[
0\r\text{f.l.} \r (a)\Dscr_{m,m+1}\otimes_X \Dscr_{m+1,n}\r (a)\Dscr_{m,n}\r 0
\]
from which we find
$H^1(X,(a)\Dscr_{m,n})=H^1(X,(a)(\Dscr_{m,m+1}\otimes_X \Dscr_{m+1,n}))$
by Lemma \ref{vanishing}. From \eqref{ref-10-10-2} we obtain an exact sequence
\begin{align}
\label{version1}
 \Tor_1^{o_X}((a)\Dscr_{m,m+1},\Dscr_{m+1,n})\r  \Jscr(a)\otimes_X \Dscr_{m+1,n} \oplus (a-1)\Dscr_{m+1,n} &\r \\
\notag D_{m.m+1}\otimes_k (a)\Dscr_{m+1,n} \r(a)\Dscr_{m,m+1}\otimes_X \Dscr_{m+1,n} &\r 0 
\end{align}
One deduces again, for example using \cite[Theorem 5.5.10]{VdB19}, that 
$\Tor_1^{o_X}((a)\Dscr_{m,m+1},\Dscr_{m+1,n})$ has finite length.  
It is clear that $(a-1)\Dscr_{m+1,n}$ has no finite length subobjects. We claim this 
is the same for $\Jscr(a)\otimes_X \Dscr_{m+1,n}$. Indeed tensoring the short exact
sequence
\[
0\r \Jscr(a)\r \Oscr_X(a)^{\oplus 2}\r \Oscr_X(a+1)\r 0 
\]
on the right with $\Dscr_{m+1,n}$ and using $\Tor$-vanishing \cite[Theorem 8.2.1]{VdB19} we
obtain a short exact sequence
\begin{equation}
\label{short}
0\r \Jscr(a)\otimes_X \Dscr_{m+1,n}\r  (a)\Dscr_{m+1,n}^{\oplus 2}\r (a+1)\Dscr_{m+1,n}\r 0 
\end{equation}
Hence in particular $\Jscr(a)\otimes_X \Dscr_{m+1,n}\subset (a)\Dscr_{m+1,n}^{\oplus 2}$
is torsion free. We conclude that \eqref{version1} becomes in
fact a short exact sequence\\
\begin{equation}
\label{shortexact}
0\r  \Jscr(a)\otimes_X \Dscr_{m+1,n} \oplus (a-1)\Dscr_{m+1,n} \r \\
D_{m,m+1}\otimes_k (a)\Dscr_{m+1,n}\r(a)\Dscr_{m,m+1}\otimes_X \Dscr_{m+1,n} \r 0 
\end{equation}
We find that
$H^1(X,(a)(\Dscr_{m,m+1}\otimes_X\Dscr_{m+1,n}))$ is sandwiched between a
direct sum of copies of $H^1(X,(a)\Dscr_{m+1,n})$ ($= 0$ by the
induction hypothesis) and a direct sum of copies of
$H^2(X,\Jscr(a)\otimes_X \Dscr_{m+1,n})$. Now $H^2(X,\Jscr(a) \otimes_X\Dscr_{m+1,n})$ is trivial as well because it is sandwiched between a direct sum of copies of $H^2( (a-2) \Dscr_{m+1,n})$ ($=0$ by Lemma \ref{ref-7.2-11}) and $H^3(X,(a-3)\Dscr_{m+1,n})$ ($=0$ as $H^3(X,-)=0$).\\

We now prove $H^1(X,(-2)\Dscr_{m,n})=0$ when $(r,s)=(2,3)$. This can
also be done by induction on $n-m$. The case $n=m$ again follows from
the standard vanishing on $X$. For the induction step recall that for
any point $q$ there is an exact sequence:
\[ 0 \r \Oscr_X(-3) \r \Oscr_X(-2) \oplus \Oscr_X(-1) \r \Oscr_X \otimes_X m_q \r 0 \]
Applying $- \otimes_X \Dscr_{m+1,n}$ yields an exact sequence
\[ 0 \r (-3)\Dscr_{m+1,n} \r (-2)\Dscr_{m+1,n} \oplus (-1)\Dscr_{m+1,n} \r \Oscr_X \otimes_X m_q \otimes_X \Dscr_{m+1,n} \r 0 \]
where the injectivity of $(-3)\Dscr_{m+1,n} \r (-2)\Dscr_{m+1,n} \oplus (-1)\Dscr_{m+1,n}$ is a torsion/torsion free
argument as above in the derivation of \eqref{shortexact}.
%
%
%
In particular we can consider a long exact sequence of cohomology groups. As in this sequence $H^1(X,(-2)\Dscr_{m,n})=H^1(X,m_{\tau^{-m-1}p}\otimes_X \Dscr_{m+1,n})$ is sandwiched between $H^1(X,(-2)\Dscr_{m+1,n}) \oplus H^1(X,(-1)\Dscr_{m+1,n})$ and $H^2(X,(-3)\Dscr_{m+1,n})$ we can conclude by the induction hypothesis, the case $a \geq -1$ which was already done and Lemma \ref{ref-7.2-11}.
\end{proof}
We may now draw some conclusions. 
\begin{lemma} $D$ is generated in degree one.
\end{lemma}
\begin{proof} We need to show for $n>m$ that 
\[
\Gamma(X,\Dscr_{m,m+1})\otimes_k
  \Gamma(X,\Dscr_{m+1,n})\r \Gamma(X,\Dscr_{m,n})
\]
 is surjective. 

The kernel of $\Dscr_{m,m+1}\otimes_X \Dscr_{m+1,n}\r\Dscr_{m,n}$ has
finite length  whence by Lemma \ref{vanishing}:
$\Gamma(X,\Dscr_{m,m+1} \otimes_X
\Dscr_{m+1,n})\r \Gamma(X,\Dscr_{m,n})$ is surjective. 
 Hence it is sufficient to show that 
\[
\Gamma(X,\Dscr_{m, m+1}) \otimes_k \Gamma(X,\Dscr_{m+1,n})\r
\Gamma(X,\Dscr_{m,m+1}\otimes_X \Dscr_{m+1,n})
\]
is surjective. In case $(r,s)=(3,2)$ we tensor 
\eqref{ref-10-10} for $a=0$ on the right with $\Dscr_{m+1,n}$. This give
\[
\Tor_1^{o_X}((0)\Dscr_{m,m+1},\Dscr_{m+1,n})\r (-1)\Dscr_{m+1,n} \r D_{m,m+1}\otimes_k (0)\Dscr_{m+1,n} \r (0)\Dscr_{m,m+1}\otimes_X \Dscr_{m+1,n}\r 0 
\]
Since $\Tor_1^{o_X}((0)\Dscr_{m,m+1},\Dscr_{m+1,n})$ has finite length and $(-1)\Dscr_{m+1,n}$
has no finite length submodules the leftmost arrow is zero.\\
Hence we must show that
$H^1(X,(-1)\Dscr_{m+1,n})=0$. This follows from Lemma \ref{ref-7.3-12}.

In case $(r,s)=(2,3)$ by \eqref{ref-10-10-2} by a similar argument
this amounts to showing that  $H^1(X,(-1)\Dscr_{m+1,n})=0$ and  $H^1(X,\Jscr(0) \otimes_X
\Dscr_{m+1,n})=0$. The first of these claims follows from Lemma \ref{ref-7.3-12}. For the
second of these claim we invoke the definition of $\Jscr$ (see \eqref{jscr}). It follows
that we have to show
$H^1(X,(-2)\Dscr_{m+1,n})= 0$ and $H^2(X,(-3)\Dscr_{m+1,n})= 0$ and these 
are known to hold by Lemma \ref{ref-7.3-12} and Lemma \ref{ref-7.2-11}.
\end{proof}
Our next result is that $D$  has the ``correct'' Hilbert function. That
is
\begin{equation}
\dim D_{m,m+a}=
\begin{cases}
\frac{(a+1)(a+2)}{2}&\text{if $a\ge 0$}\\
0&\text{if $a<0$}
\end{cases}
\label{ref-11-13} 
\end{equation}
The case $a<0$ is trivial so we consider $a\ge 0$.
For this we have to check the cases $(r,s)=(3,2)$ and $(r,s)=(2,3)$ separately. For the quadratic case a computation similar to \cite[Corollary\ 5.2.4]{VdB19} tells us that the colength of
$\Dscr_{m,m+a}$ inside $o_X(2a)$ is
\[
3\frac{a(a+1)}{2}
\]
Using the fact that $H^1(X,\Dscr_{m,n})=0$ by Lemma \ref{ref-7.3-12} we
obtain (for $a\ge 0$)
\[
\dim D_{m,m+a}= \frac{(2a+1)(2a+2)}{2}-3\frac{a(a+1)}{2} = \frac{(a+1)(a+2)}{2}
\]
Similarly in the cubic case the colength of
$\Dscr_{m,m+a}$ inside $o_X(2a)$ is
\[
\frac{a(a+1)}{2}
\]
and again using the fact that $H^1(X,\Dscr_{m,n})=0$ we
obtain (for $a\ge 0$)
\[
\dim D_{m,m+a}= \frac{(2a+2)^2}{4}-\frac{a(a+1)}{2} = \frac{(a+1)(a+2)}{2}
\]
Hence in both cases \eqref{ref-11-13} holds.

Finally we prove the following.
\begin{lemma}
The canonical map $D\r D_Y$ is surjective.
\end{lemma}
\begin{proof}
  As $D$ and $D_Y$ are both generated in degree 1 (for $D_Y$ this is
  proved in the same way as for $B(Y,\sigma,\Lscr)$, see
  \cite{ATV1}), it suffices to check that $D_{m,m+1} \rightarrow
  (D_Y)_{m,m+1}$ is surjective.
For this consider the following commuting diagram (with $\Oscr_X(-Y)=\Oscr_X\otimes_{o_X} o_X(-Y)$ the subobject of $\Oscr_X$ corresonding to the ideal
$gA\subset A$)
\begin{center}
\begin{tikzpicture}
\label{ref-tikz-surj}
\matrix(m)[matrix of math nodes,
row sep=3em, column sep=3em,
text height=1.5ex, text depth=0.25ex]
{ & 0 & 0 & 0 & \\
 & \Oscr_X(-Y) & \Oscr_X \otimes_X m_d & \Oscr_X \otimes_X m_{d,Y} & \\
0 & \Oscr_X(-Y) & \Oscr_X & \Oscr_Y & 0 \\
0 & 0 & \Oscr_d & \Oscr_d & 0 \\
 & 0 & 0 & 0 &  \\};
\path[->,font=\scriptsize]
(m-1-2) edge (m-2-2)
(m-1-3) edge (m-2-3)
(m-1-4) edge (m-2-4)
(m-3-1) edge (m-3-2)
(m-4-1) edge (m-4-2)
(m-2-2) edge (m-2-3)
 edge (m-3-2)
(m-2-3) edge (m-3-3)
 edge (m-2-4)
(m-2-4) edge (m-3-4) 
(m-3-2) edge (m-3-3)
 edge (m-4-2)
(m-3-4) edge (m-4-4)
 edge (m-3-5) 
(m-3-3) edge (m-3-4)
 edge (m-4-3)
(m-4-2) edge (m-5-2)
 edge (m-4-3)
(m-4-3) edge (m-5-3)
 edge (m-4-4)
(m-4-4) edge (m-5-4)
 edge (m-4-5);
\end{tikzpicture}
\end{center}
The bottom two rows and the first column are obviously exact. The third column is equal to
\[ 0 \rightarrow \Mscr_{d,Y} \rightarrow \Oscr_Y \rightarrow \Oscr_d \rightarrow 0 \]
and hence is exact. The exactness of the middle column follows as usual from \cite[Lemma 8.2.1]{VdB19}.
Hence we can apply the Snake lemma to the above diagram and find the following exact sequence:
\[ 0 \rightarrow \Oscr_X(-Y) \rightarrow \Oscr_X \otimes_X m_d \rightarrow \Oscr_X \otimes_X m_{d,Y} \rightarrow 0 \]
As the above obviously remains true when we replace $d$ by $\sigma^{-m}d$ and as $o_X(2)$ is an invertible bimodule we get an exact sequence
\[ 0 \rightarrow \Oscr_X(-Y) \otimes_X o_X(2) \rightarrow \Oscr_X
\otimes_X \Dscr_{m,m+1} \rightarrow \Oscr_X \otimes_X (\Dscr_Y)_{m,m+1}
\rightarrow 0 \] The surjectivity of $D_{m,m+1} \rightarrow
(D_Y)_{m,m+1}$ then follows from $H^1(X,\Oscr_X(-Y) \otimes_X o_X(2))=H^1(X, \Oscr_X(-1))=0$ using that
$o_X(-Y)=o_X(-3)$ (see \S\ref{secconstruction}) as well as the standard vanishing results for $X$
(see \cite[Theorem 8.1]{artinzhang}).
%
\end{proof}
Since now the map $D\r D_Y$ is surjective, one checks using
\eqref{ref-11-13} that $D_{m,n}\r D_{Y,m,n}$ is an isomorphism
for $n\le m+2$. Thus $D$ and $D_Y$ have the same quadratic relations.
Let $D'$ be the quadratic AS-regular $\ZZ$-algebra associated to
$(Y,\Lscr_0,\Lscr_1)$ (see \S\ref{secremas}). Then since $D'$ is
quadratic, and has the same quadratic relations as $D_Y$ we obtain a
surjective map $D'\r D$.  Since $D'$ and $D$ have the same Hilbert
series by \eqref{ref-11-13} we obtain $D'\cong D$. Hence $D$ is the
quadratic AS-regular $\ZZ$-algebra associated to
$(Y,\Lscr_0,\Lscr_1)$.

\section{Non-commutative function fields}
As above let $A$ be a $3$-dimensional Sklyanin algebra, which may be either quadratic or cubic,
with geometric data $(Y,\sigma,\Lscr)$ and let $D$ be the AS-regular $\ZZ$-subalgebra
of $\check{A}^{(2)}$ constructed in \S\ref{secconstruction}.

Let $A'$ be the 3-dimensional quadratic Sklyanin algebra with geometric data
$(Y,\sigma,\Lscr_0)$ if $A$ is quadratic and $(Y,\psi,\Lscr_0)$ if $A$
is cubic where $(\Lscr_i)_i$ is as in \eqref{lidef}.

 By the discussion at the end of \S\ref{secremas}
together with Lemma \ref{geometricdatalemma} we conclude that $D\cong \check{A}'$.


We will
now show that there is an isomorphism between  the
function fields of $\Proj A$ and $\Proj A'$. In the case that
$A$ is cubic this will be the final step in the proof of Theorem \ref{mainth1}.
If $A$ is quadratic then the relation between $A$ and $A'$ is a generalization
of the classical Cremona transform.

\medskip

By the graded version of Goldie's Theorem \cite[Corollary
8.4.6.]{nastasescumethods} the non-zero homogeneous elements of $A$
form an Ore set $S$ and hence the graded field of fractions
$A[S^{-1}]$ of $A$ exists. By the structure theorem for graded fields
\cite{NVO1} it is of the form
\[
\Frac(A) = \Frac_0(A) [t,t^{-1},\alpha]
\]
where $\Frac_0(A)$ is a division algebra concentrated in degree zero,
$|t|=1$ and $\alpha$ is an automorphism $\alpha: \Frac_0(A)
\rightarrow \Frac_0(A): a \mapsto tat^{-1}$. $\Frac_0(A)$ was
introduced in the introduction as ``the function field'' of $\Proj A$.
Our aim is to show that $\Frac_0(A)\cong \Frac_0(A')$.

It is straightforward to generalize the concept of an Ore set and its
corresponding localization to
$\ZZ$-algebras. In fact this is the classical concept of an Ore set
in a category (and its corresponding localization).

If $S\subset A$ is a multiplicative closed Ore set consisting of
homogeneous elements then one defines a corresponding multiplicative
closed Ore set $\check{S}\subset \check{A}$ by putting $\check{S}_{ij}=S_{j-i}$. 
%
%
A straightforward 
check yields $\widecheck{A[S^{-1}]} \cong \check{A}[\check{S}^{-1}]$.

Now let $S$ and $S'$ be the set of nonzero homogeneous elements in $A$
respectively $A'$. Then the inclusion $\check{A'} \hookrightarrow
\check{A}^{(2)} \hookrightarrow \check{A}$ restricts to $\check{S'}
\hookrightarrow \check{S}$ and hence for arbitrary $i \in \ZZ$ there is an
induced map $\zeta_i$:
\[ \Frac_0(A') = \left( A'[(S')^{-1}] \right)_0 \cong \left(
  \check{A'}[\check{S'}^{-1}] \right)_{ii} \rightarrow \left(
  \check{A}[\check{S}^{-1}] \right)_{2i,2i} \cong \left( A[S^{-1}]
\right)_0 = \Frac_0(A)\] 
Although 
this map depends on
$i$  we will show that it is always
an isomorphism.

\medskip

As $\Frac_0(A')$ and $\Frac_0(A)$ are division rings and $\zeta_i \neq
0$, it follows that $\zeta_i$ is always injective, so the only nontrivial thing to do,
is proving its surjectivity. So given any $a,s \in
\check{A}_{2i,2j_1} \backslash \{0 \}$ we need to find a $j_2 \in \ZZ$
and $h \in \check{A}_{2j_1,2j_2}$ such that $ah, sh \in
\check{A'}_{i,j_2}$. We claim we can find such an $h$ only depending
on $n := j_1 - i$ and not on $a$ or $s$. For this consider the
following map:
\[ 
\Gamma \big( X,o_X(2n) \big) \otimes \Gamma \big( X,o_X(2N) \otimes_X \Iscr \big) \rightarrow \Gamma \big( X, o_X(2(n+N)\otimes_X \Iscr) \big) 
\]
where 
$\Iscr$ is the ideal in $o_X$ such that $o_X(2(n+N)) \otimes_X \Iscr=\Dscr_{i_1,i_1+n+N}$ (see \eqref{case2}).

If we can choose an $N$ such that 
\begin{equation}
\label{diamondsuit} \dim_k  \Gamma \big( X,o_X(2N) \otimes_X \Iscr \big)  \neq \{ 0 \} 
\end{equation}
then there is an element $0 \neq h \in \check{A}_{2i+2n,2i+2n+2N}$ and an embedding
\[ \check{A}_{2i,2i+2n} \hookrightarrow \check{A'}_{i,i+n+N}: a \mapsto
a h \] which yields the surjectivity of $\Frac_0(A') \rightarrow
\Frac_0(A)$ (as we may take $j_2=j_1+N=i+n+N$ in the above). So it suffices to prove \eqref{diamondsuit}. As the cases
$(r,s)=(3,2)$ and $(r,s)=(2,3)$ are completely similar, we only treat
the first case.

Note that the codimension of $\Gamma \big( X,o_X(2N) \otimes_X
\Iscr \big)$ inside $\check{A}_{2i_1+2n,2i+2n+2N}$ is at most $3
\frac{N(N+1)}{2}$ which grows like
$\frac{3N^2}{2}$.  On the other hand $\dim_k \left(
  \check{A}_{2i_1+2n,2i_1+2n+2N} \right) = \frac{(2N+1)(2N+2)}{2}$
which grows like $2N^2$, so for $N$ sufficiently large
\eqref{diamondsuit} will be fulfilled.

\section{Relation with non-commutative blowing up}
\label{secblowup}
For the interested reader we now sketch how the $\ZZ$-algebra $D$ which was introduced in a somewhat adhoc
manner in \S\ref{secconstruction} may be obtained in a natural way from the formalism of non-commutative
blowing up as introduced in \cite{VdB19}.

First let us remind the reader how the commutative Cremona transform works. Let $p_1,p_2,p_3\in \PP^2$ be three distinct non-collinear
points on $\PP^2$ and consider the linear system of quadrics passing through those points. This is a $5-3=2$ dimensional
linear system and hence it defines a birational transformation $\phi:\xymatrix@1{\PP^2\ar@{.>}[r]&\PP^2}$ with $\{p_1,p_2,p_3\}$
being the points of indeterminacy.

The indeterminacy of $\phi$ may be resolved by blowing up the points $\{p_1,p_2,p_3\}$. Let $\alpha:\tilde{X}\r \PP^2$ be the 
resulting surface and let $L_1,L_2,L_3$ be the exceptional curves. Then the Cremona transform factors as
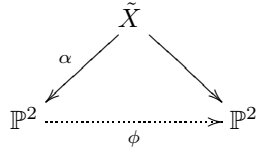
\[
\xymatrix{
&\tilde{X}\ar[dr]\ar[dl]_{\alpha}\ \\
\PP^2\ar@{.>}[rr]_{\phi}&& \PP^2
}
\]
where the right most map is obtained from the  sections of the line bundle $\Oscr_{\tilde{X}}(1)=\alpha^\ast(\Oscr_{\PP^2}(2))\otimes_{\tilde{X}}\Oscr_{\tilde{X}}(-L_1-L_2-L_2)$
on $\tilde{X}$.

\medskip

Now we replace $\PP^2$ by the non-commutative $X$ given by $\Proj A$ where $A$ is a 3-dimensional quadratic Sklyanin algebra. 
We will use again the standard notation $Y,\Lscr,\sigma,p_1,p_2,p_3,d,\ldots$. According to \cite{VdB19}
we may blow up\footnote{In \cite{VdB19} we discuss the case of a blowup of a single
  point. Blowing up a set of points is similar.} $X$ in $d$ to obtain a map of non-commutative schemes $\alpha:\tilde{X}\r X$. Then we need a substitute for the line bundle  $\Oscr_{\tilde{X}}(1)$ on $\tilde{X}$.
Actually in the non-commutative case it is more natural to look for a substitute for the family of objects $(\Oscr_{\tilde{X}}(n))_n$ since then there is an associated $\ZZ$-algebra
\[
\bigoplus_{m,n}D_{m,n}=\bigoplus_{m,n}\Hom_{\tilde{X}}(\Oscr_{\tilde{X}}(-n),\Oscr_{\tilde{X}}(-m))
\]
This idea has been used mainly in the case that the sequence is ample in a suitable sense (e.g.\ \cite{Polishchuk}), but  the associated $\ZZ$-algebra may be defined
in general. Of course in the non-ample case the relation between the sequence and the underlying non-commutative scheme will be weaker.

\medskip

Let us now carry out this program in somewhat more detail. According to \cite{VdB19} we have $\tilde{X}=\Proj \Dscr$ where $\Dscr$ is a graded
algebra in $\Bimod(X-X)$ given by
\[
o_X\oplus m_d(Y)\oplus m_{d}m_{\tau^{-1}d}(2Y)\oplus\cdots\oplus m_{d}\cdots m_{\tau^{-n+1}d}(nY)\oplus\cdots 
\]
The inclusion $o_X\r \Dscr$ yields the map $\alpha:\tilde{X}\r X$.

Suitable noncommutative analogue of the objects $\Oscr_{\tilde{X}}(-mL_1-mL_2-mL_2)$ turn
out to be the objects in $\Bimod(X-\tilde{X})$ associated to the $o_X-\Dscr$-bimodules
given by
\[
m_{\tau^m d} \cdots m_{\tau d}\oplus m_{\tau^m d} \cdots m_d(Y)\oplus
m_{\tau^m d} \cdots m_{\tau^{-1}d}(2Y)\oplus\cdots \oplus m_{\tau^m d} \cdots  m_{\tau^{-n+1}d}(nY)\oplus\cdots 
\]
Up to right bounded $o_X-\Dscr$-bimodules (which are invisible in Proj) these are the same as
\[
(o_X(-mY)\otimes_X \Dscr)[m]
\]
where $[1]$ is the shift functor on $\Dscr$-modules (or bimodules, or variants thereof). 
So the non-commutative analogues of the objects $\Oscr_{\tilde{X}}(n)$ turn out to be
associated to 
\[
(\Oscr_X(2n-nY)\otimes_X \Dscr)[n]
\]
where we have written $O_X(a+bY)$ for $O_X(a)\otimes_X o_X(bY)$.

Or ultimately 
\[
\Oscr_{\tilde{X}}(n)=\alpha^\ast(\Oscr_X(2n-nY))[n]
\]

We now compute (the fourth equality requires an argument similar to \cite[Proposition 8.3.1(2)]{VdB19})
\begin{align*}
D_{m,n}&=\Hom_{\tilde{X}}(\Oscr_{\tilde{X}}(-n),\Oscr_{\tilde{X}}(-m))\\
 &= \Hom_{\tilde{X}}(\alpha^\ast(\Oscr_X(-2n+nY))[-n],
  \alpha^\ast(\Oscr_X(-2m+mY))[-m])\\
&=\Hom_{X}(\Oscr_X(-2n+nY),
  \alpha_\ast(\alpha^\ast(\Oscr_X(-2m+mY))[n-m]))\\
&=\Hom_{X}(\Oscr_X(-2n+nY),\Oscr_X(-2m+mY)\otimes_{o_X} \Dscr_{n-m})\\
  &=\Hom_{X}(\Oscr_X(-2n),\Oscr_X(-2m)m_{\tau^{-m}d}\cdots
  m_{\tau^{-n+1}d}))\\
  &=\Gamma(X,\Oscr_X(-2m)m_{\tau^{-m}d}\cdots m_{\tau^{-n+1}d} \otimes o_X(2n))
\end{align*}
Hence we find indeed the same result as in \S\ref{secconstruction}.
%
%
\bibliographystyle{plain}
\bibliography{Veronesereferences}
\end{document}